\documentclass[a4paper,12pt]{article}
\usepackage{amsmath,amsthm,amssymb,fancybox,longtable,graphics,amscd,multirow}
\usepackage[dvipdfm]{graphicx}
\usepackage{caption,setspace}
\usepackage{hyperref}
\usepackage{color}

\newtheorem{thm}{\textsc{Theorem}}[section]
\newtheorem*{thma}{\textsc{Theorem A}}
\newtheorem*{thmb}{\textsc{Theorem B}}

\newtheorem{lem}{\textsc{Lemma}}[section]
\newtheorem{cor}{\textsc{Corollary}}[section]

\def\QED{$\Box$}

\def\mbi#1{\boldsymbol{#1}} 

\def\Pic{\mathop{\mathrm{Pic}}\nolimits}
\def\TPic#1{\mathop{(\mathrm{Pic}_{#1})_{\rm tor}}\nolimits}

\def\rk{\mathop{\mathrm{rank}}\nolimits}
\def\sign{\mathop{\mathrm{sgn}}\nolimits}

\def\WP{\mathbb{WP}}
\def\CC{\mathbb{C}}

\newtheorem{remark}{\textsc{Remark}}

\theoremstyle{definition}
\newtheorem{defn}{\textsc{Definition}}[section]

\begin{document}
\title{Mirror constructions for K3 surfaces from bimodal singularities}
\author{Makiko Mase and Ursula Whitcher \thanks{The second author would like to thank the Isaac Newton Institute for Mathematical Sciences, Cambridge, for support and hospitality during the K-theory, algebraic cycles and motivic homotopy theory programme, where work on this paper was undertaken. This work was supported by EPSRC grant no EP/R014604/1. Both authors are indebted to the anonymous referees for helpful comments that improved the readability of this paper.}}
\maketitle

\section{Introduction}\label{S:Intro}

Given a K3 surface realized as a hypersurface in a weighted projective space or a Gorenstein Fano toric variety, one may construct a mirror K3 surface in various ways.  Depending on the precise model, available descriptions of mirror symmetry include the Greene--Plesser mirror, the Berglund-H\"{u}bsch transpose construction for invertible polynomials, 
Dolgachev-Nikulin-Pinkham's lattice-polarized K3 surfaces, 
and Batyrev's reflexive polytope construction.  The multitude of descriptions raises the question of whether mirror constructions are consistent.  Comparing different mirror constructions often entails making choices---one might need to specify a family containing a K3 surface or a lattice polarization, for example---and thus it is important to establish systematic methods for making these choices.

Ebeling and Ploog studied invertible polynomials in three variables obtained from Arnold's classification of bimodal singularities in \cite{EP}; we list these polynomials in Table~\ref{Ta:bimodal}.  

\begin{table}[ht]
\centering
\begin{tabular}{|c|c|c|c|}
\hline
Name &  $f$   & $f^T$ &  Dual \\
\hline
$J_{3,0}$ & $x^6+xy^3+z^2$ & $x^6y+y^3+z^2$ & $Z_{13}$ \\
$Z_{1,0}$ & $x^5y + xy^3 +z^2$ & $x^5y + xy^3 +z^2$ & $Z_{1,0}$ \\
$Q_{2,0}$  & $x^4z + xy^3 +z^2$ &  $x^4y + y^3 + xz^2$& $Z_{17}$\\
$W_{1,0}$  & $x^6+y^2z+z^2$ & $x^6+y^2+yz^2$  & $W_{1,0}$ \\
$S_{1,0}$  & $x^5y+y^2z+z^2$ & $x^5+xy^2+yz^2$ & $W_{17}$ \\
$U_{1,0}$  & $x^3y+y^2z+z^3$ &  $x^3+xy^2+yz^3$  & $U_{1,0}$ \\
\hline
$E_{18}$  & $x^5 + y^3 + xz^2$ & $x^5z + y^3 + z^2$  & $Q_{12}$ \\
$E_{19}$  & $x^7 + xy^3+z^2$  & $x^7y + y^3+z^2$  & $Z_{1,0}$  \\
$E_{20}$ & $x^{11} +y^3 +z^2$ &  $x^{11} +y^3 +z^2$ & $E_{20}$ \\
\hline
$Z_{18}$  & $x^6y + xy^3 + z^2$ &  $x^6y + xy^3 + z^2$ & $Z_{18}$  \\
$Z_{19}$   & $x^9y + y^3 + z^2$ & $x^9 + xy^3 + z^2$ & $E_{25}$\\
\hline
$Q_{16}$  & $x^4z + y^3 + xz^2$ &  $x^4z + y^3 + xz^2$ & $Q_{16}$ \\
$Q_{17}$ & $x^5z+xy^3+z^2$  & $x^5y+y^3+xz^2$  & $Z_{2,0}$  \\
$Q_{18}$  & $x^8z+y^3+z^2$ & $x^8+y^3+xz^2$ & $E_{30}$  \\
\hline
$W_{17}$ & $x^5+xz^2+y^2z$ &  $x^5z+yz^2+y^2$  & $S_{1,0}$  \\
$W_{18}$ & $x^7+ y^2z+ z^2$ &  $x^7+ y^2+ yz^2$ & $W_{18}$  \\
\hline
$S_{16}$ & $x^4y+xz^2+y^2z$ &  $x^4y+xz^2+y^2z$  & $S_{16}$  \\
$S_{17}$ & $x^6y+y^2z+z^2$ &  $x^6+xy^2+yz^2$ & $X_{2,0}$  \\
\hline
$U_{16}$  & $x^5+y^2z+yz^2$ & $x^5+y^2z+yz^2$ & $U_{16}$  \\
\hline
\end{tabular}
\caption{Strange duality of the bimodal singularities} \label{Ta:bimodal}
\end{table}

In a series of papers, the first author and her collaborators have compared mirror constructions for K3 surfaces obtained by extending the bimodal singularity polynomials to an invertible polynomial in four variables.  In \cite{MU}, they observed there is an extension to an invertible polynomial defining a K3 surface in weighted projective space for all but 4 examples, showed Berglund-H\"{u}bsch duality for these K3 surface invertible polynomials can be viewed as a special case of Batyrev mirror symmetry, and commented on potential relationships with homological mirror symmetry. These observations are consistent with broader efforts to unify Berglund-H\"{u}bsch and Batyrev mirror symmetry, such as the construction of Clarke in \cite{Clarke}; the relationship between Berglund-H\"{u}bsch, Batyrev, and homological mirror symmetry is treated in depth in \cite{FK} and \cite{DFK}.

In \cite{Mase2016a, Mase2016b, Mase2017}, the first author studied 
Dolgachev-Nikulin-Pinkham's
mirror symmetry construction for K3 surfaces obtained from bimodal singularities using an invertible polynomial in four variables.  Let $\Lambda_{K3} \cong U \oplus U \oplus U \oplus E_8 \oplus E_8$ be the K3 lattice, the unique (up to isomorphism) unimodular lattice of signature $(3,19)$; here $U$ is the unimodular lattice of signature $(1,1)$ and we take $E_8$ to be negative definite.  According to \cite{Dolgachev}, if two K3 surfaces $X$ and $\check{X}$ are mirror, there should exist lattice polarizations $M \hookrightarrow H^2(X,\mathbb{Z})$ and $\check{M} \hookrightarrow H^2(\check{X},\mathbb{Z})$  such that $M^\perp \cong \check{M} \oplus nU$, where $n$ is a positive integer and the perpendicular complement is computed using the isomorphism $H^2(X,\mathbb{Z}) \cong \Lambda_{K3}$.  In particular, $\rk M + \rk \check{M}$ must be equal to $20$.  Given a reflexive polytope $\Delta$, we may obtain a K3 surface $X$ as a hypersurface in the toric variety $\mathbb{P}_{\!\Delta}$ obtained by an appropriate resolution of singularities of the normal fan of $\Delta$.  In this case, the obvious lattice polarization to choose is $\Pic_\Delta$, the lattice generated by intersecting $X$ with the divisors of $\mathbb{P}_{\!\Delta}$.  The main results of \cite{Mase2016a, Mase2016b, Mase2017} identify polar dual pairs of reflexive polytopes $\Delta$ and $\Delta^\circ$ and associated K3 hypersurfaces in Gorenstein Fano toric varieties which satisfy 
Dolgachev-Nikulin-Pinkham-style mirror symmetry using the lattices $\Pic_\Delta$ and $\Pic_{\Delta^\circ}$ for all but five of the bimodal singularity pairs studied in \cite{MU}. 

Our first Main Theorem shows that the remaining five pairs of bimodal singularity polynomials cannot be extended to invertible polynomials whose Newton polytopes are polar dual reflexive polytopes $\Delta$ and $\Delta^\circ$ such that $\Pic_\Delta$ and $\Pic_{\Delta^\circ}$ yield mirror lattice polarizations in the sense of Dolgachev-Nikulin-Pinkham. 

\begin{thma}\label{T:main}
For each of the bimodal singularity mirror pairs $(B',\, B)$ being
\[
(Z_{13},\, J_{3,0}),\, 
(X_{2,0},\, S_{17}),\, 
(W_{18},\, W_{18}),\, 
(W_{17},\, S_{1,0}),\, 
(U_{16},\, U_{16}),\,
\]
let $f$ be the defining equation of $B$. 
Then, there does not exist an invertible deformation $F$ of $f$ such that the Newton polytope of $F$ is a reflexive polytope $\Delta$ and $\rk \Pic_\Delta + \rk \Pic_{\Delta^\circ}=20$. 
\end{thma}

\begin{remark}
Let $\Delta$ be a reflexive polytope obtained as the Newton polytope of a four-variable invertible deformation of one of the bimodal singularities listed in Theorem~A, let $\mathbb{P}_{\!\Delta}$ be the smooth toric variety determined by a maximal projective subdivision of the normal fan of $\Delta$, and let $X$ be a regular K3 hypersurface in $\mathbb{P}_{\!\Delta}$. 
The map 
\[
r: H^{1,1}(\mathbb{P}_{\!\Delta}) \to H^{1,1}({X})
\]
is a natural restriction of the Hodge $(1,1)$-components. 

The rank of the cokernel of $r$ is known to be bounded by a \emph{toric correction term} representing divisors of $\mathbb{P}_{\!\Delta}$ whose intersection with $X$ has multiple components:
\[
\mathrm{coker}(r) \geq \sum_\Gamma \ell^*(\Gamma)\ell^*(\Gamma^\circ),
\]
where the sum is over all edges in $\Delta$, $\ell^*(\Gamma)$ represents the number of lattice points in the relative interior of an edge, and $\Gamma^\circ$ is the polar dual of $\Gamma$. 
By a direct computation, in the course of the proof of the Main Theorem we will show that the toric correction term is nonzero for all $\Delta$ associated to bimodal singularities from the Main Theorem. 
Therefore the main theorem can be rephrased as the statement that we cannot choose a reflexive Newton polytope for the pairs 
\[
(Z_{13},\, J_{3,0}),\, 
(X_{2,0},\, S_{17}),\, 
(W_{18},\, W_{18}),\, 
(W_{17},\, S_{1,0}),\, 
(U_{16},\, U_{16}),\,
\]
such that the natural restriction map $r$ of the Hodge $(1,1)$-components is surjective. 
\end{remark}

We can relax the requirements of Theorem~A in two ways.  First, it might happen that the Newton polytope associated to an invertible polynomial is simply not reflexive.  In such cases, one may include the Newton polytope in a larger reflexive polytope; the first author has pursued this strategy in past work, including~\cite{MU}.  A reflexive Newton polytope may also admit inclusion in a larger reflexive polytope, yielding a different polarizing lattice.

Alternatively, we may fix our choice of reflexive polytope but choose the polarizing lattice more carefully.  This strategy has previously been pursued in order to resolve the apparent contradiction between the Dolgachev-Nikulin-Pinkham
and Batyrev mirror symmetry constructions for K3 hypersurfaces in toric varieties.  A proposal introduced by Rohsiepe in the preprint \cite{Rohsiepe} and reviewed in \cite{WhitcherSurvey} describes an appropriate choice of polarizing lattices for K3 surfaces realized as hypersurfaces in Gorenstein Fano toric varieties obtained from reflexive polytopes.  The idea is to consider a sublattice $\TPic{\Delta}$ of $\Pic_\Delta$ given by the so-called \emph{toric divisors}; these are the divisors given by the pullback of the divisors of the ambient toric variety. The \emph{toric correction term} measures the difference between the rank of $\Pic_\Delta$ and the rank of $\TPic{\Delta}$. As we will discuss in Section~\ref{S:mirror}, the toric correction term is symmetric, so it also measures the difference between the rank of $\Pic_{\Delta^\circ}$ and the rank of $\TPic{\Delta^\circ}$.

%

\begin{thm}[\cite{Rohsiepe}]\label{T:Rohsiepe}
Let $\Delta$ and $\Delta^\circ$ be polar dual three-dimensional reflexive polytopes.  Then $(\Pic_\Delta)^\perp \cong \TPic{\Delta^\circ}\oplus U$ and $(\TPic{\Delta})^\perp \cong \Pic_{\Delta^\circ} \oplus U$.
\end{thm}

Note that because the rank of the unimodular lattice $U$ is 2 and the rank of the K3 lattice is 22, we have $\rk{(\Pic_{\Delta})} + \rk{(\Pic_{\Delta^\circ})} \geq 20$, with equality if and only if the toric correction term vanishes.

It follows from Theorem~\ref{T:Rohsiepe} that any of the bimodal singularities studied in \cite{MU} may be extended to a polynomial defining a K3 hypersurface in a Gorenstein Fano toric variety in such a way that Batyrev mirror symmetry for the hypersurface induces lattice mirror symmetry in the sense of Dolgachev-Nikulin-Pinkham. 
Our second Main Theorem shows that we may make these choices in a way compatible with Berglund-H\"{u}bsch duality.

\begin{thmb}\label{T:main2} Let $(B',\, B)$ be one of the bimodal singularity mirror pairs
\[
(Z_{13},\, J_{3,0}),\, 
(X_{2,0},\, S_{17}),\, 
(W_{18},\, W_{18}),\, 
(W_{17},\, S_{1,0}),\, 
(U_{16},\, U_{16}),\,
\] 
and let $f$ be the defining equation of $B$. 
Then there exists a reflexive polytope $\Delta$ and an invertible deformation $F$ of $f$ such that the Newton polytope of $F$ is a subpolytope of $\Delta$, the Newton polytope of $F^T$ is a subpolytope of $\Delta^\circ$, and $(\Pic_\Delta)^\perp \cong \TPic{\Delta^\circ}\oplus U$.
\end{thmb}

The plan of the paper is as follows.  In Section~\ref{S:mirror} we review the various mirror constructions, remark on the connection between Rohsiepe's mirror proposal and Morrison's monomial-divisor mirror map, and establish notation.  In Section~\ref{S:MainTheorem} we prove Theorem~A by a case-by-case analysis.  In Section~\ref{S:Rohsiepe} we prove Theorem~B and describe an explicit choice of polytopes and lattices in each case. Some of our analysis uses computations in the computer algebra systems Mathematica and SageMath. See the ancillary files attached to the preprint version of this paper, \cite{MWpreprint}, for illustrative code.

\section{Mirror constructions}\label{S:mirror}

In this section, we review the mirror constructions we are comparing and establish notation.

\subsection{Berglund-H\"{u}bsch duality}

We briefly review the Berglund-H\"{u}bsch duality construction in the form we will use it.  For more detail see, for example, \cite{ABS}.  Let $A=(a_{ij})$ be a matrix of nonnegative integers, and consider the corresponding polynomial $F_A$ given by the sum of $n+1$ monomials in $n+1$ variables:

\[F_A = \sum_{i=0}^n \prod_{j=0}^n x_j^{a_{ij}}.\]

\begin{defn}
We say the polynomial $F_A$ is \emph{invertible} if the matrix $A$ is invertible, there exist positive integers $q_j$ such that $\sum_{j=0}^n q_j a_{ij}$ is the same constant for all $i$, and the polynomial $F_A$ has exactly one critical point, at the origin.
\end{defn}

Invertible polynomials were classified in \cite{KS}.  They may be written as sums of polynomials in disjoint variables of three \emph{atomic types}: Fermat polynomials, of the form $x_1^{a_1}$, loops, of the form $x_1^{a_1}x_2 + x_2^{a_2}x_3+\dots+x_{m-1}^{a_{m-1}}x_m + x_m^{a_m}x_1$, and chains, of the form $x_1^{a_1}x_2 + x_2^{a_2}x_3+\dots+x_{m-1}^{a_{m-1}}x_m + x_m^{a_m}$.

\begin{defn}
We say an invertible polynomial $F_A$ satisfies \emph{the Calabi-Yau condition} if $\sum_{j=0}^n q_j a_{ij} = \sum_{j=0}^n q_j$.
\end{defn}

If $F_A$ is an invertible polynomial satisfying the Calabi-Yau condition, then the weights $(q_0, \dots, q_n)$ determine a weighted projective space $\WP(q_0, \dots, q_n)$ and $F_A$ determines a Calabi-Yau hypersurface $X_A$ in this weighted projective space.  When $n=3$, $X_A$ is a K3 surface.

In general, Berglund-H\"{u}bsch duality is a duality of orbifolds.  Given an appropriately chosen group of discrete symmetries $G$ acting on a Calabi-Yau hypersurface $X_A$, we obtain a mirror pair $X_A/G$ and $X_{A^T}/G^T$, where $X_{A^T}$ is the Calabi-Yau hypersurface determined by the transpose matrix $A^T$ and $G^T$ can be computed from the data of $G$ and $A$.  We will focus on the case where $G$ is trivial.  In this case, $G^T$ consists of automorphisms of $X_{A^T}$ that are induced by the multiplicative action of the torus $(\CC^*)^{n+1}$ on the ambient weighted projective space and act \emph{symplectically} on $X_{A^T}$, preserving the holomorphic form.

\subsection{Batyrev's duality and the monomial-divisor mirror map}

We review the mirror symmetry construction for Calabi-Yau hypersurfaces in toric varieties described in \cite{BatyrevMirror}, examine the differences that arise in the case of K3 surfaces, and establish notation.  For a more detailed exposition, see \cite{CoxKatz} for Calabi-Yau varieties or \cite{WhitcherSurvey} for the K3 surface case.  Let $N \cong \mathbb{Z}^k$ be an integral lattice, and let $M = \mathrm{Hom}(N, \mathbb{Z})$ be the dual lattice.  Let $N_\mathbb{R} = N \otimes \mathbb{R}$ and $M_\mathbb{R} = M \otimes \mathbb{R}$ be the corresponding vector spaces.  The duality between $N$ and $M$ induces a real-valued pairing $\langle v, w \rangle$ between elements of $N_\mathbb{R}$ and $M_\mathbb{R}$.  

A \emph{lattice polytope} in $N_\mathbb{R}$ or $M_\mathbb{R}$ is the finite hull of a convex set of vectors in the lattice.  Given a lattice polytope $\Delta$ in $M_\mathbb{R}$ containing the origin, we define its \emph{polar dual} $\Delta^\circ$ in $N_\mathbb{R}$ as
\[\{v \in N \mid \langle v, w \rangle \geq -1 \text{ for all } w \in \Delta\}.\]
If $\Delta^\circ$ is also a lattice polytope, we say $\Delta$ and $\Delta^\circ$ are \emph{reflexive}.  Note that $(\Delta^\circ)^\circ = \Delta$.  Reflexive polytopes in three and four dimensions are classified in \cite{KSReflexive}: up to lattice isomorphism, there are 4,319 reflexive polytopes in three dimensions and 473,800,776 in four dimensions.  The authors of \cite{KSReflexive} also specify a normal form that selects a unique representative of each equivalence class.

Now, assume $\Delta$ is a reflexive polytope in $M_\mathbb{R}$.  We may obtain a fan $R$ in $N_\mathbb{R}$ by taking the normal fan to $\Delta$, or, equivalently, by taking the fan over the faces of $\Delta^\circ$.  A fan obtained in this way determines a Gorenstein Fano toric variety.  A simplicial refinement $\Sigma$ of the fan $R$ such that the one-dimensional cones are the lattice points of $\Delta^\circ$ is called a \emph{maximal projective subdivision} of the fan.  In three dimensions, a maximal projective subdivision $\Sigma$ determines a smooth toric variety, which we call $\mathbb{P}_\Delta$.  In four dimensions, the resulting toric variety $\mathbb{P}_\Delta$ may have orbifold singularities.  However, in either case, a general anticanonical hypersurface in such a toric variety will be smooth.  

The hypersurfaces obtained using three-dimensional reflexive polytopes are smooth K3 surfaces, while four-dimensional reflexive polytopes determine Calabi-Yau threefolds.  Such a hypersurface will be a semiample divisor in the ambient toric variety $\mathbb{P}_\Delta$, but need no longer be ample.  Given a maximal projective subdivision $\Sigma$ of the normal fan to $\Delta$ and a hypersurface $X$ in $\mathbb{P}_\Delta$, we say $X$ is \emph{$\Sigma$-regular}, or, if the context is clear, \emph{regular}, if the intersection of $X$ with the torus $\mathbf{T}_\sigma$ is either empty or a smooth subvariety of codimension 1 for every cone $\sigma$ in $\Sigma$. 

Given a face $\Gamma$ of a lattice polytope, we write $\ell(\Gamma)$ for the number of lattice points in $\Gamma$, and $\ell^*(\Gamma)$ for the number of lattice points in the relative interior of $\Gamma$.  In \cite{BatyrevMirror}, for $k \geq 4$, Batyrev derived the following formulas for the Hodge numbers of a regular Calabi-Yau $k-1$-fold $X$ obtained from an $n$-dimensional reflexive polytope $\Delta \subset M_\mathbb{R}$ by the above procedure:

\begin{align}
h^{1,1}(X) &= \ell(\Delta^\circ) - k - 1 - \sum_{\mathrm{codim}\; \Gamma^\circ = 1} \ell^*(\Gamma^\circ) + \sum_{\mathrm{codim} \; \Gamma^\circ = 2} \ell^*(\Gamma^\circ) \ell^*(\Gamma)\\
h^{k-2,1}(X) &= \ell(\Delta) - k - 1 - \sum_{\mathrm{codim}\; \Gamma = 1} \ell^*(\Gamma) + \sum_{\mathrm{codim} \; \Gamma^\circ = 2} \ell^*(\Gamma^\circ) \ell^*(\Gamma)
\end{align}

\noindent By reversing the roles of $\Delta$ and $\Delta^\circ$, we obtain formulas for the Hodge numbers of the mirror $k-1$-fold $X^\circ$; we see $h^{1,1}(X) = h^{k-2,1}(X^\circ)$ and $h^{k-2,1}(X) = h^{1,1}(X^\circ)$.  

Batyrev's computation highlights the role of two important subspaces.  The so-called \emph{toric} subspace $H^{1,1}_{\rm tor}(X)$ of $H^{1,1}(X)$ is given by the pullback of $H^{1,1}(\mathbb{P}_\Delta)$ along the natural inclusion map, and has dimension $\ell(\Delta^\circ) - k - 1 - \sum_{\mathrm{codim}\; \Gamma^\circ = 1} \ell^*(\Gamma^\circ)$.  The \emph{polynomial deformation space} $H^{k-2,1}_{\rm poly}(X)$ is a subspace of $H^{k-2,1}(X)$ isomorphic to the first-order polynomial deformations of $X$, and has dimension $\ell(\Delta) - k - 1 - \sum_{\mathrm{codim}\; \Gamma = 1} \ell^*(\Gamma)$.  We call $\sum_{\mathrm{codim} \; \Gamma^\circ = 2} \ell^*(\Gamma^\circ) \ell^*(\Gamma)$, which measures the difference between $H^{1,1}_{\rm tor}(X)$ and $H^{1,1}(X)$ or $H^{k-2,1}_{\rm poly}(X)$ and $H^{k-2,1}(X)$, the \emph{toric correction term}.  Aspinwall, Greene, and Morrison showed in \cite{AGM} that there is a natural isomorphism between $H^{1,1}_{\rm tor}(X)$ and $H^{k-2,1}_{\rm poly}(X^\circ)$, induced by a correspondence between divisors and monomials which they termed the \emph{monomial-divisor mirror map}.

In the case of a K3 hypersurface in a three-dimensional toric variety, we have $k-2 = 1$.  Since all K3 surfaces have the same Hodge diamond, the equality of Hodge numbers is trivial.  As Dolgachev observed in \cite{Dolgachev}, and as we will discuss in more detail in \S~\ref{S:Dolgachev}, one may instead study the structure of the Picard group of a K3 surface.  For a regular K3 surface $X$ obtained from a three-dimensional reflexive polytope $\Delta \subset M_\mathbb{R}$, we have the following inequality (see \cite{Rohsiepe} for a detailed discussion):

\begin{equation}\label{E:PicInequality}
\mathrm{rank}\,\mathrm{Pic}(X) \geq \ell(\Delta^\circ) - 4 - \sum_{\mathrm{codim}\; \Gamma^\circ = 1} \ell^*(\Gamma^\circ) + \sum_{\mathrm{codim} \; \Gamma^\circ = 2} \ell^*(\Gamma^\circ) \ell^*(\Gamma).
\end{equation}

Here, $\ell(\Delta^\circ) - 4 - \sum_{\mathrm{codim}\; \Gamma^\circ = 1} \ell^*(\Gamma^\circ)$ measures the rank of the subgroup of so-called \emph{toric divisors} $(\mathrm{Pic}_\Delta)_\mathrm{tor}$, generated by the pullback of the divisors in the ambient space. The toric correction term measures the rank of a sublattice $L_0(\Delta)$ corresponding to divisors of the ambient space whose intersection with $X$ splits into multiple components.  Together, $\mathrm{Pic}_\mathrm{tor}(X)$ and $L_0(\Delta)$ generate the lattice that we shall call $\mathrm{Pic}_\Delta$. We denote the rank of this lattice by $\rho_\Delta$. We have the formula

\begin{equation}\label{E:rhoFormula}
\rho_\Delta = \ell(\Delta^\circ) - 4 - \sum_{\mathrm{codim}\; \Gamma^\circ = 1} \ell^*(\Gamma^\circ) + \mathrm{rank}\,L_0(\Delta).
\end{equation}

The inequality $\mathrm{rank}\,\mathrm{Pic}(X) \geq \rho_\Delta$ is inevitable, because non-isotrivial families of K3 surfaces do not have constant Picard rank: Oguiso showed in \cite{OguisoPic} that any analytic neighborhood in the base of a one-parameter,
non-isotrivial family of K3 surfaces has a dense subset where the Picard ranks of
the corresponding surfaces are strictly greater than the minimum Picard rank of that family.  On the other hand, Bruzzo and Grassi show in \cite{Bruzzo-Grassi} that when the K3 hypersurfaces are both smooth and ample, $\mathrm{Pic}(X)=(\mathrm{Pic}_\Delta)_\mathrm{tor}$ for very general $X$.  In this case, the toric correction term is 0, so $(\mathrm{Pic}_\Delta)_\mathrm{tor} = \mathrm{Pic}_\Delta$. (Recall that a very general property holds outside a countable union of proper closed subvarieties).  

In our comparison of mirror constructions, we often begin with a polynomial $f$ and then explore different ways that this polynomial could be used to represent a hypersurface in a toric variety. Combinatorially, the polynomial corresponds to a choice of lattice points determined by its monomials, and we are choosing different ways to include these points in a lattice polytope (typically a reflexive polytope). Hypersurfaces obtained in this way need not be isomorphic. However, by \cite[Theorem 1.1]{DFK}, as long as $f$ is not divisible by one of its coordinates $x_i$, they are birational.

\subsection{Reconciling Batyrev and 
Dolgachev-Nikulin-Pinkham mirror symmetry}\label{S:Dolgachev}

Dolgachev proposed in \cite{Dolgachev} that one should view mirror symmetry for K3 surfaces as a relationship between moduli spaces of \emph{lattice-polarized} K3 surfaces.  Recall that a \emph{lattice} is a finitely-generated free $\mathbb{Z}$-module equipped with a non-degenerate, symmetric integer-valued form.  Let $E_8$ be the unimodular, negative definite $ADE$ lattice of rank 8, and let $U$ be the indefinite unimodular \emph{hyperbolic lattice}, with intersection form given by $\left(\begin{smallmatrix} 0 & 1 \\ 1 & 0 \end{smallmatrix}\right)$.  Then for any K3 surface $X$, there exists an isomorphism $\phi: H^2(X,Z) \to \Lambda_{K3}$, where $\Lambda_{K3}:= U \oplus U \oplus U \oplus E_8 \oplus E_8$ is the \emph{K3 lattice}.  A choice of such a $\phi$ is called a \emph{marking}.  

Suppose we are given an even, nondegenerate lattice $M$ of signature $(1,t)$ and a primitive embedding $M \hookrightarrow \Lambda_{K3}$.  (Recall that a lattice embedding is primitive if the quotient of the ambient lattice by the image of the embedding is a free abelian group.)  We say $X$ is \emph{$M$-polarized} if there exists a primitive embedding $M \hookrightarrow \mathrm{Pic}(X)$, and we say $X$ is \emph{marked $M$-polarized} if there exists a marking $\phi$ such that $\phi^{-1}(M) \subseteq \mathrm{Pic}(X)$.  In \cite{Dolgachev}, Dolgachev constructed moduli spaces of marked $M$-polarized K3 surfaces satisfying an appropriate pseudo-ampleness condition.

We say $M$ is $m$-\emph{admissible} if $M^\perp = J \oplus \check{M}$ for some lattice $J$ isomorphic to $mU$.  In this situation, we call $\check{M}$ the \emph{mirror} of $M$.  Note that $\mathrm{rank} M + \mathrm{rank}\check{M} = 20$.  The simplest case, and the one we will be concerned with in what follows, is when $M$ is 1-admissible.  When this happens, $(\check{M})^\perp = J \oplus M$, so not only is $\check{M}$ the mirror of $M$, but $M$ is the mirror of $\check{M}$.  Furthermore, one may view the moduli spaces of $M$- and $\check{M}$-polarized K3 surfaces as mirror families.

Now, suppose we are given a family of K3 surfaces realized as hypersurfaces in a toric variety obtained from a reflexive polytope.  To study 
Dolgachev-Nikulin-Pinkham's 
mirror symmetry construction in this context, one must choose a polarizing lattice.  The choice is simplest when the toric correction term is 0.  When we calculate the toric correction term for each of the 4319 isomorphism classes of three-dimensional reflexive polytopes, we find that 1863 yield a toric correction term of 0, of which 53 correspond to self-dual reflexive polytopes; thus, the 0 toric correction term case is common, but not the majority.  Given a reflexive polytope $\Delta$ in one of these 1863 isomorphism classes, we let $M$ be the space of divisors induced by the ambient toric variety, $\mathrm{Pic}_\Delta = (\mathrm{Pic}_\Delta)_\mathrm{tor}$.  Rohsiepe showed in \cite{Rohsiepe} that in this case $M^\perp$ is isomorphic to $U \oplus \mathrm{Pic}_{\Delta^\circ}$, so $\check{M}$ is $\mathrm{Pic}_{\Delta^\circ} = (\mathrm{Pic}_{\Delta^\circ})_\mathrm{tor}$.  This choice of polarizing lattice $M$ is clearly symmetric: if we swap the roles of $\Delta$ and $\Delta^\circ$, we also swap the roles of $M$ and $\check{M}$.  We may view this computation as a confirmation that for toric K3 hypersurfaces, the monomial-divisor mirror map is not merely an identification of groups, but extends to the lattice structure.  

For the remaining 2456 isomorphism classes, we cannot set $M$ to $(\mathrm{Pic}_\Delta)_\mathrm{tor}$ and $\check{M}$ to $(\mathrm{Pic}_{\Delta^\circ})_\mathrm{tor}$ because the ranks of the lattices are too small, and we cannot set $M$ to $\mathrm{Pic}_\Delta$ and $\check{M}$ to $\mathrm{Pic}_{\Delta^\circ}$ because the ranks of those lattices are too big.  However, a computation described in \cite{Rohsiepe} shows that if we choose $M$ to be $(\mathrm{Pic}_\Delta)_\mathrm{tor}$, then $M^\perp$ is always isomorphic to $U \oplus \mathrm{Pic}_{\Delta^\circ}$.  Thus, we may reconcile the Batyrev and 
Dolgachev-Nikulin-Pinkham 
 mirror symmetry constructions if we polarize by the pullback of the divisors of the ambient space on one side, and the full lattice generated by intersecting our K3 surface with the ambient space divisors on the other side.

\subsection{Picard lattices and intersection numbers}\label{S:intersection}

In Section~\ref{S:Rohsiepe}, we illustrate the computation of $\mathrm{Pic}_\Delta$ and $(\mathrm{Pic}_\Delta)_\mathrm{tor}$ for specific cases of interest. We will use the following facts about lattices: 

\begin{cor}[Corollary 1.6.2~\cite{Nikulin80}]\label{orthogonal}
Let lattices $S$ and $T$ be primitively embedded into the $K3$ lattice. 
Then $S$ and $T$ are orthogonal to each other in the $K3$ lattice if and only if $q_S\simeq -q_T$, where $q_S$ (resp. $q_T$) is the discriminant form of $S$ (resp. $T$). 
\end{cor}

\begin{cor}[Corollary 1.12.3~\cite{Nikulin80}]\label{primitive}
Let $S$ be an even lattice of signature $(t_+, t_-)$ and $\Lambda$ be an even unimodular lattice of signature $(l_+, l_-)$. 
There exists a primitive embedding of $S$ into $\Lambda$ if and only if the following three conditions are simultaneously satisfied. 
\begin{enumerate}
\item[$(1)$] $l_+-l_- \equiv 0 \mod 8$, 
\item[$(2)$] $l_- -t_- \geq 0$ and $l_+ -t_+ \geq 0$, and 
\item[$(3)$] $\rk \Lambda - \rk S > l(A_S)$. 
\end{enumerate}
Here $A_S$ denotes the discriminant group of $S$, which is finitely-generated abelian, and $l(A_S)$ is the minimal number of generators of $A_S$. 
\end{cor}

We will use the following fundamental lemma. 
\begin{lem}\label{PrimeDet}
Let $L$ be a sublattice of the $K3$ lattice $\Lambda_{K3}:=U^{\oplus3}\oplus E_8^{\oplus2}$. 
If the discriminant number $\det(L)$ of $L$ is prime, then, $L$ has no proper overlattices. 
\end{lem}
\begin{proof} Let $\det(L)=p$ be a prime number. 
If it were an overlattice $L\subset N\subsetneq \Lambda_{K3}$ (the $N$ is an overlattice), then, we would have $[N:L]^2\det{(N)} = \det{(L)} = 1\cdot p$. 
Thus, one has $[N:L]=1$. 
Therefore, there does not exist such a lattice $N$. 
In particular, the discriminant group $A_L$ of $L$ is isomorphic to $\mathbb{Z}\slash p\mathbb{Z}$. \end{proof}

We will use the following formulas to compute the Picard lattices from toric data. 

Let $\mathbf{M}$ be a three-dimensional lattice. 
Let $\{\mathbf{e}_1,\, \mathbf{e}_2,\, \mathbf{e}_3\}$ be the standard basis of $\mathbb{R}^3\simeq \mathbf{M}\otimes_{\mathbb{Z}}\mathbb{R}=: \mathbf{M}_\mathbb{R}$. Let $\Delta$ be a reflexive polytope in $\mathbf{M}_\mathbb{R}$. 
Denote by ${\rm Div}_{\mathbb{T}}(\widetilde{\mathbb{P}}_{\!\Delta})$ the set of all toric divisors $\widetilde{D}_i$ in the MPCP toric resolution $\widetilde{\mathbb{P}}_{\!\Delta}$ of the projective toric variety $\mathbb{P}_{\!\Delta}$~\cite{BatyrevMirror}. The lattice $\mathbf{M}$ is explicitly given in Section~3, using the description given in~\cite{MU}. It is defined by the weight system $\mbi{w}=(w_0,\,w_1,\,w_2,\,w_3)$ such that the deformation $F$ of the defining polynomial $f$ of a bimodal singularity is an anticanonical member of the weighted projective space of weight system $\mbi{w}$. 

The toric divisors are related by a linear system with three equations: 
\begin{equation}
\sum_{i=1}^{d+3}(v_i,\, \mathbf{e}_j) \widetilde{D}_i= 0 \qquad j=1,2,3, \label{ToricLinearRelation}
\end{equation}
where $( \, , \, )$ is the standard inner product in $\mathbb{R}^3$. 

If one restricts a toric divisor $\widetilde{D}_i\in{\rm Div}_{\mathbb{T}}(\widetilde{\mathbb{P}}_{\!\Delta})$ to the minimal model $\widetilde{X}$ of a generic section $X$ by the MPCP toric resolution of the ambient space, one has a divisor on $\widetilde{X}$ which we denote by $D_i$. 

As to the intersection numbers, 
\begin{equation}\label{SelfIntersection}
D_i^2 = 
\begin{cases}
2\ell^*(\psi_i)-2 & \textnormal{if $v_i$ is a vertex. }\\
-2 & \textnormal{if $v_i$ is in the interior of an edge. }
\end{cases}
\end{equation}
Here $\psi_i$ is the face in $\Delta^*$ that is dual to $v_i$, and $\ell^*(\psi_i)$ is the number of lattice points in the interior of the face $\psi_i$. 
Moreover,  
\begin{equation}\label{Intersection}
D_i.D_j = 
\begin{cases}
1 & \textnormal{if $v_i$ and $v_j$ are next to each other on an edge. } \\
\ell^*(m_{ij}^*)+1 & \textnormal{if $v_i$ and $v_j$ are vertices that are connected } \\
 & \textnormal{by an edge $m_{ij}$ whose dual is $m_{ij}^*$. } \\
0 & \textnormal{otherwise}. 
\end{cases}
\end{equation}
For the above formulas $(2), (3)$, see \cite[Chapter 5]{FultonToric}; for further illustrations of their use in the case of K3 hypersurfaces, see \cite{MaseDis} or \cite{Rohsiepe}. 


\section{Invertible deformations}\label{S:MainTheorem}

In the following subsections, we prove our first Main Theorem by a case-by-case analysis:
 
\begin{thma}
For each of the bimodal singularity mirror pairs $(B',\, B)$ being
\[
(Z_{13},\, J_{3,0}),\, 
(X_{2,0},\, S_{17}),\, 
(W_{18},\, W_{18}),\, 
(W_{17},\, S_{1,0}),\, 
(U_{16},\, U_{16}),\,
\]
let $f$ be the defining equation of $B$. 
Then, there does not exist an invertible deformation $F$ of $f$ such that the Newton polytope of $F$ is a reflexive polytope $\Delta$ and $\rk \Pic_\Delta + \rk \Pic_{\Delta^\circ}=20$. 
\end{thma}

\subsection{$Z_{13}$ and $J_{3,0}$}\label{S:J30}

The singularity $J_{3,0}$ is defined by $f = x^6 + xy^3 + z^2$. 

First, we choose a basis 
\[
e_1 = (-1,5,-1,-1),\, 
e_2 = (-1,0,2,-1),\, 
e_3 = (-1,-1,-1,1)
\]
of the group
\[
\mathbf{M}:= \left\{ (i,j,k,l)\in\mathbb{Z}^4 \mid i+3j+5k+9l =0  \, {\, \rm and \, }\, k+l\equiv 0 \mod 2\right\}. 
\]
Then, the monomials $X^6$, $XY^3$, and $Z^2$ correspond to the lattice points $(1,0,0)$, $(0,1,0)$, and $(0,0,1)$, respectively. 

On the other hand, for the singularity of type $Z_{13}$ defined by $f^T=x^6y+y^3+z^2$, choose a basis 
\[
e_1'=(-2,1,0,0),\,
e_2'=(-6,0,1,0),\,
e_3'=(-9,0,0,1)
\]
of the group
\[
\mathbf{M}':= \left\{ (i,j,k,l)\in\mathbb{Z}^4 \mid i+2j+6k+9l=0\right\}. 
\]
Then the monomials $X^6Y$, $Y^3$, and $Z^2$ correspond to the lattice points
$(5,0,-1)$, $(-1,2,-1)$, $(-1,-1,1)$, respectively. 


The possible choices of $F$ for $J_{3,0}$ are given by adding the monomial $W^9Z$ or $W^{18}$ to the polynomial $f$, since $F$ should be an invertible polynomial. 

Case (i): In the case that the deformation is $F=X^6 +XY^3 +Z^2 +W^9Z$, the Newton polytope of $F$ is not reflexive. 
Indeed, the Newton polytope of $F$ contains a facet 
\[
\Phi = {\rm Conv}\{ (1,0,0),\, (0,1,0),\, (-1,-3,-4)\}. 
\]
However, the polar dual of the facet $\Phi$ is easily computed to be the non-integral vertex $(-1,\,  -1 ,\, 5/4)\not\in\mathbb{Z}^3$. 

Case (ii): In the case that the deformation is $F=X^6+XY^3+Z^2+W^{18}$, the Newton polytope of $F$ is reflexive. 
Indeed, the Newton polytope of $F$ is the convex hull 
\[
\Delta = {\rm Conv}\left\{ (1,0,0),\, (0,1,0),\, (0,0,1),\, (-2,-6,-9)\right\}. 
\]
Then, the polar dual polytope is given by 
\[
\Delta^\circ = {\rm Conv}\left\{ (8,-1,-1),\, (-1,2,-1),\, (-1,-1,1),\, (-1,-1,-1)\right\}. 
\]
Since $\Delta^\circ$ is an integral polytope, the polytope $\Delta$ is reflexive.  In the \cite{sage} database of reflexive polytopes, $\Delta$ and $\Delta^\circ$ are index 745 and 4282, respectively.  The first author and a collaborator showed in \cite{MU} that $\Delta^\circ$ contains the Newton polytope of $F^T$ as a subpolytope.

Next, we compute the toric contribution $\rk L_0(\Delta)$. 

Let $\Gamma$ be the edge of $\Delta$ given by 
\[
\Gamma = {\rm Conv} \left\{ (0,0,1),\, (-2,-6,-9) \right\}. 
\]
One has its polar dual
\[
\Gamma^\circ = {\rm Conv}\left\{(-1,2,-1),\, (8,-1,-1) \right\}. 
\]
No other edge contributes to $\rk L_0(\Delta)$.  Thus, 
\[
\rk L_0(\Delta) = 2. 
\]

\subsection{\bf $X_{2,0}$ and $S_{17}$}\label{SS:invertibleX20}

The singularity $S_{17}$ is defined by $f = x^6y + y^2z + z^2$. 

First, we choose a basis 
\[
e_1 = (1,1,-1,0),\, 
e_2 = (-6,0,1,1),\, 
e_3 = (-12,0,0,3)
\]
of the group
\[
\mathbf{M}:= \left\{ (i,j,k,l)\in\mathbb{Z}^4 \mid i+j+2k+4l =0 \text{ and } j-i\equiv 0 \mod 3\right\}. 
\]
Then, the monomials $X^6Y$, $Y^2Z$, and $Z^2$ correspond to the lattice points $(5,5,-2)$, $(-1,0,0)$, and $(-1,-2,1)$, 
respectively. 

On the other hand, for the singularity of type $X_{2,0}$ defined by $f^T=x^6+xy^2+yz^2$, choose a basis 
\[
e_1'=(0,5,-1,-1),\,
e_2'=(-1,5,0,-2),\,
e_3'=(0,-2,0,1)
\]
of the group 
\[
\mathbf{M}':= \left\{ (i,j,k,l)\in\mathbb{Z}^4 \mid i+j+3k+2l=0\right\}. 
\]
Then, the monomials $XY^2$, $YZ^2$, and $WX^6$ correspond to the lattice points $(-1,1,0)$, $(0,1,3)$, and $(1,0,0)$,
respectively. 

The only possible choice of $F$ is given by adding the monomial $W^7X$ to the polynomial $f$, since $F$ should be an invertible polynomial. 
Note that this monomial corresponds to the lattice point $(0,-1,0)$. 
The Newton polytope of $F$ is not reflexive. 
Indeed, the Newton polytope of $F$ contains a facet 
\[
\Phi = {\rm Conv}\left\{ (5,5,-2),\, (-1,0,0),\, (0,-1,0)\right\}. 
\]
However, the polar dual of the facet $\Phi$ is easily computed to be the non-integral vertex $(1,\, 1,\, 11/2)\not\in\mathbb{Z}^3$. 

Similarly, the Newton polytope of $F^T$ is not reflexive. 
Indeed, the Newton polytope of $F^T$ contains a facet 
\[
\Phi' = {\rm Conv}\left\{ (-1,1,0), \, (0,1,3), \, (1,0,0)\right\}. 
\]
However, the polar dual of the facet $\Phi'$ is easily computed to be the non-integral vertex $(-1,\ -2,\ 1/3)\not\in\mathbb{Z}^3$.

\subsection{\bf $W_{18}$}\label{SS:invertibleW18}

The singularity $W_{18}$ is defined by $f = x^7 + y^2z + z^2$.  This singularity is dual to $f^T = x^7 + y^2 + yz^2$, another realization of $W_{18}$.

First, we choose a basis 
\[
e_1 = (0,0,2,-1),\, 
e_2 = (1,1,1,-1),\, 
e_3 = (6,-1,0,-1)
\]
of the group
\[
{\mathbf{M}:=} \left\{ (i,j,k,l)\in\mathbb{Z}^4 \mid 3i+4j+7k+14l =0 \right\}. 
\]
Then the monomials $X^7$, $Y^2Z$, and $Z^2$ correspond to the lattice points 
$(-3,5,-1)$, $(1,-1,0)$, and $(0,-1,0)$,
respectively. 

On the other hand, choose a basis 
\[
e'_1 = (-1,-3,1,0),\, 
e'_2 = (1,5,-1,-1),\, 
e'_3 = (1,-1,0,0)
\]
of the group
\[
{\mathbf{M}':=} \left\{ (i,j,k,l)\in\mathbb{Z}^4 \, |\, i+j+4k+2l =0 \right\}. 
\]
In this basis, the monomials $WX^7$, $Y^2$, and $YZ^2$ correspond to the lattice points $(0,1,-1)$, $(2,1,0)$, and $(-1,-1,-1)$
respectively. 

The possible choices of $F$ are given by adding the monomial $W^7Y$ or $W^8X$ to the polynomial $f$, since $F$ should be an invertible polynomial. 

Case (i): 
In the case that the deformation is $F=X^7+Y^2Z+Z^2+W^7Y$, the Newton polytope of $F$ is not reflexive. 
Indeed, the Newton polytope of $F$ contains a facet
\[
\Phi = {\rm Conv}\left\{ (-3,5,-1), \, (1,-1,0),\, (0,0,1)\right\}. 
\]
However, the polar dual of the face $\Phi$ is easily computed to be the non-integral vertex $(-7/2,\, -5/2,\, -1)\not\in\mathbb{Z}^3$. 

Case (ii): 
In the case that the deformation is $F=X^7+Y^2Z+Z^2+W^8X$, the Newton polytope of $F$ is reflexive. 
Indeed, the Newton polytope of $F$ is the convex hull  
\[
\Delta = {\rm Conv}\left\{ (-3,5,-1), \, (1,-1,0),\, (-1,1,1),\, (0,-1,0)\right\}. 
\]
Then, the polar dual polytope is given by 
\[
\Delta^\circ = {\rm Conv}\left\{ (0,1,-2), \, (0,1,6),\, (2,1,0),\, (-4,-3,-2)\right\}. 
\]
Since $\Delta^\circ$ is an integral polytope, the polytope $\Delta$ is reflexive.  In the \cite{sage} database, $\Delta$ and $\Delta^\circ$ are polytopes 9 and 4312, respectively.  

Next, we compute the toric contribution $\rk L_0(\Delta)$. Let $\Gamma$ be the edge of $\Delta$ given by 
\[
\Gamma = {\rm Conv}\left\{ (-1,1,1),\, (-3,5,-1)\right\}. 
\]
The polar dual of $\Gamma$ is
\[
\Gamma^\circ = {\rm Conv}\left\{ (2,1,0),\, (-4,-3,-2)\right\}, 
\]
and 
\[
\ell^*(\Gamma) = \ell^*(\Gamma^\circ) = 1. 
\]
No other edge contributes to $\rk L_0(\Delta)$. 
Thus, 
\[
\rk L_0(\Delta) = 1. 
\]

In this case, the polytope $\Delta^T$ corresponding to $F^T$ has vertices $(0, 1, -1)$, $(2, 1, 0)$, $(-1, -1, -1)$, and $(0,1,6)$.  This polytope is not reflexive.  Furthermore, one may use the normal form for three-dimensional lattice polytopes described in \cite{KS} and implemented in \cite{sage} to check that $\Delta^T$ is not isomorphic to any lattice sub-polytope of $\Delta^\circ$.  Thus, for this choice of $\Delta$, Batyrev's duality does not appear consistent with Berglund-H\"{u}bsch duality.  In \cite{MU}, the first author and her collaborator showed how to choose a reflexive polytope containing $\Delta$ that resolves this discrepancy.  We will study that polytope in more detail in Section~\ref{S:Rohsiepe}.

\subsection{\bf $W_{17}$ and $S_{1,0}$} 

The singularity $S_{1,0}$ is defined by $f = x^5y + y^2z + z^2$.  

First, we choose a basis 
\[
e_1 = (0,5,-1,-1),\, 
e_2 = (-1,4,0,-1),\, 
e_3 = (-1,-1,-1,1)
\]
of the group
\[
{\mathbf{M}:=} \left\{ (i,j,k,l)\in\mathbb{Z}^4 \mid 2i+3j+5k+10l =0 \right\}. 
\]
Then, the monomials $X^5Y$, $Y^2Z$, and $Z^2$ correspond to the lattice points 
$(0,1,0)$, $(-1,1,0)$, and $(0,0,1)$,  
respectively. 

On the other hand, choose a basis 
\[
e'_1 = (4,0,-1,0),\, 
e'_2 = (-6,1,1,0),\, 
e'_3 = (-3,0,0,1)
\]
of the group
\[
{\mathbf{M}':=} \left\{ (i,j,k,l)\in\mathbb{Z}^4 \mid i+2j+4k+3l =0 \right\}. 
\]
Then the monomials $X^5$, $XY^2$, and $YZ^2$ correspond to the lattice points 
$(5,4,-1)$, $(-1,0,-1)$, and $(-1,-1,1)$ 
respectively. 

The only possible choice of $F$ is given by adding the monomial $W^{10}$ to the polynomial $f$, since $F$ should be an invertible polynomial. 
Note that on this side the monomial $W^{10}$ corresponds to the lattice point $(4,-6,-3)$. 
However, with this choice the Newton polytope of $F$ is not reflexive. 
Indeed, the Newton polytope of $F$ contains the facet
\[
\Phi = {\rm Conv}\left\{ (-1,1,0),\, (0,1,0),\, (4,-6,-3)\right\}. 
\]
However, the polar dual of the face $\Phi$ is easily computed to be the non-integral vertex $(
0,\, -1,\, 7/3)\not\in\mathbb{Z}^3$.

\bigskip
We consider whether we can obtain a reflexive Newton polytope by working on the other side, using the polynomial $f^T = X^5 + X Y^2 + Y Z^2$ defining the singularity $W_{17}$. There are two possibilities for a projectivization $G$ of $f^T$, namely, $G=f^T+W^{10}$ or $G=f^T+W^7Z$.
We claim that in both cases, the Newton polytope of $G$ is not reflexive. 
Note that the monomials $W^{10}$ and $W^7Z$ are respectively corresponding to lattice points $(0,-1,-1)$ and $(0,-1,0)$. 

If a projectivization is $G=f^T+W^{10}$, the dual of the facet 
\[
{\rm Conv}\{ (-1,  0, -1),
(-1, -1,  1),
( 0, -1, -1) \} = {\rm Conv}\{XY^2, YZ^2, W^{10} \} 
\]
in the Newton polytope of $G$ is the vertex $1/3(2,2,1)\not\in\mathbb{Z}^3$. 
Thus, the Newton polytope of $G$ is not reflexive.

If a projectivization is $G=f^T+W^7Z$, the dual of the facet
\[
{\rm Conv}\{(-1,0,-1),(5,4,-1),(0,-1,0)\}={\rm Conv}\{XY^2, X^5, W^7Z\}
\]
in the Newton polytope of $G$ is the vertex $1/3(-2, 3, 5)\not\in\mathbb{Z}^3$. 
Thus, the Newton polytope of $G$ is not reflexive.

\subsection{\bf $U_{16}$} 

The singularity $U_{16}$ is defined by $f = x^5 + y^2z + yz^2$. This singularity is self-dual.
First, we choose a basis 
\[
e_1 = (4,-1,-1,0),\, 
e_2 = (4,-1,0,-1),\, 
e_3 = (5,0,-1,-1)
\]
of the group
\[
{\mathbf{M}:=} \left\{ (i,j,k,l)\in\mathbb{Z}^4 \mid 2i+3j+5k+5l =0 \right\}. 
\]
Then the monomials $X^5$, $Y^2Z$, and $YZ^2$ correspond to the lattice points 
$(-2,-2,3)$, $(0,1,-1)$, and $(1,0,-1)$,
respectively. 

On the other hand, choose a basis 
\[
e'_1 = (0,-2,0,1),\, 
e'_2 = (0,-2,1,0),\, 
e'_3 = (1,3,-1,-1)
\]
of the group
\[
{\mathbf{M}':=} \left\{ (i,j,k,l)\in\mathbb{Z}^4 \mid i+j+2k+2l =0 \right\}. 
\]
Then the monomials 
$WX^5$, $Y^2Z$, and $YZ^2$ correspond to the lattice points 
$(-1,-1,0)$, $(-1,0,-1)$, and $(0,-1,-1)$,
respectively. 

The only possible choice of $F$ is given by adding the monomial $W^6X$ to the polynomial $f$ since $F$ should be an invertible polynomial; thus, we have $F=X^5Y + Y^2Z + Z^2 + W^6X$.
Note that the monomial $W^6X$ corresponds to the lattice point $(0,0,1)$. 

The Newton polytope of $F$ is the polytope $\Delta$ given by:
\[
\Delta = {\rm Conv}\left\{  (1,0,-1), \, (0,0,1),\, (-2,-2,3),\,  (0,1,-1) \right\}. 
\]
The polar dual polytope of $\Delta$ is given by
\[
\Delta^\circ = {\rm Conv}\left\{ (-2,-2,-1),\, (-2,1,-1),\, (4,4,5),\, (1,-2,-1)\right\}. 
\]
Since the polytope $\Delta^\circ$ is an integral polytope, the polytope $\Delta$ is reflexive. The index of $\Delta$ in the \cite{sage} database is 1, and the index of $\Delta^\circ$ is 4281.

Next, we compute the toric contribution $\rk L_0(\Delta)$. 

Let $\Gamma$ be the edge in $\Delta$ given by
\[
\Gamma = {\rm Conv}\left\{ (0,0,1),\, (-2,-2,3)\right\}. 
\]
The polar dual edge is
\[
\Gamma^\circ = {\rm Conv}\left\{ (-2,1,-1),\, (1,-2,-1)\right\},
\]
and 
\[
l^*(\Gamma) = 1,\qquad l^*(\Gamma^\circ)=2. 
\]
No other edge contributes to $\rk L_0(\Delta)$. 
Thus, 
\[
\rk L_0(\Delta) = l^*(\Gamma) l^*(\Gamma^\circ)=2. 
\]

\section{Mirror polytopes and mirror lattices}\label{S:Rohsiepe}

In this section, we examine Theorem B:

\begin{thmb} Let $(B',\, B)$ be one of the bimodal singularity mirror pairs
\[
(Z_{13},\, J_{3,0}),\, 
(X_{2,0},\, S_{17}),\, 
(W_{18},\, W_{18}),\, 
(W_{17},\, S_{1,0}),\, 
(U_{16},\, U_{16}),\,
\] 
and let $f$ be the defining equation of $B$. 
Then there exists a reflexive polytope $\Delta$ and an invertible deformation $F$ of $f$ such that the Newton polytope of $F$ is a subpolytope of $\Delta$, the Newton polytope of $F^T$ is a subpolytope of $\Delta^\circ$, and $(\Pic_\Delta)^\perp \cong \TPic{\Delta^\circ}\oplus U$.
\end{thmb}

In practice, several different polytopes $\Delta$ may be available. In each case, we describe a specific choice of $\Delta$ and verify that the lattices $\TPic{\Delta^\circ}$ and $U\oplus\Pic_\Delta$ are orthogonal to each other in the $K3$ lattice, using the formulas in Section~\ref{S:intersection}.

\subsection{$Z_{13}$ and $J_{3,0}$} 
The $J_{3,0}$ singularity is defined by  $f=x^6+xy^3+z^2$, and $Z_{13}$ by $f^T=x^6y+y^3+z^2$. 
Recall from Section~\ref{S:J30} that the singularity $J_{3,0}$ is defined by $f = x^6 + xy^3 + z^2$, and may be extended to an invertible deformation  $F=X^6+XY^3+Z^2+W^{18}$ with a reflexive Newton polytope $\Delta$ given by:
\[
\Delta = {\rm Conv}\left\{ (1,0,0),\, (0,1,0),\, (0,0,1),\, (-2,-6,-9)\right\}. 
\]
We label the lattice points on the edges of $\Delta$ as
\[
\begin{matrix}
m_1=(1,0,0), & m_2=(0,1,0), & m_3=(0,0,1), \\
m_4=(-2,-6,-9), & m_5=(0,-2,-3), & m_6=(-1,-4,-6), \\
m_7=(-1,-3,-4).
\end{matrix}.
\]

The polar dual polytope of $\Delta$ is given by
\[
\Delta^\circ = {\rm Conv}\left\{ (8,-1,-1),\, (-1,2,-1),\, (-1,-1,1),\, (-1,-1,-1)\right\}. 
\]

Label the lattice points on the edges of $\Delta^\circ$ by
\[
\begin{matrix}
v_1=(-1,-1,1), & v_2=(-1,2,-1), & v_3=(-1,-1,-1),\\
v_4=(8,-1,-1), & v_5=(-1,-1,0), & v_6=(-1,0,-1),\\ 
v_7=(-1,1,-1), & v_8=(0,-1,-1), & v_9=(1,-1,-1),\\
v_{10}=(2,-1,-1), & v_{11}=(3,-1,-1), & v_{12}=(4,-1,-1),\\
v_{13}=(5,-1,-1), & v_{14}=(6,-1,-1), & v_{15}=(7,-1,-1),\\
v_{16}=(5,0,-1), & v_{17}=(2,1,-1).
\end{matrix}.
\]
 


The Newton polytope of the deformation $F^T=X^6Y+Y^3+Z^2+W^{18}$ of $f^T=x^6y+y^3+z^2$ is a subpolytope of the polar dual $\Delta^\circ$ of $\Delta$. 

As we have seen in the previous section, the toric contribution is $\rk L_0(\Delta)=2$. 
Moreover, one has
\begin{eqnarray*}
\rho_\Delta = 4+13+2-3 = 16, & & 
\rho_{\Delta^\circ} = 4+3+2-3 = 6. 
\end{eqnarray*}

We now compute $\TPic{\Delta^\circ}$. 
The collection $\{ m_5, \, m_6,\, m_7\}$ consists of linearly independent vectors. 
Let $L$ be a lattice that is generated by divisors $D_1,\, D_2,\, D_3$ and $D_4$. 
By taking the new generators $\{ D_1,\, -D_2 + D_3,\, -2 D_1+3 D_2-2 D_3+D_4,\, D_1+D_4\}$, we have
\[
L\simeq U\oplus A_2. 
\]
By a direct computation, we have $\det L = -3$ and $\sign L = (1,3)$.

Note that the intersection matrix of $L$ with respect to the generators $\{D_1,\, D_2,\, D_3,\, D_4\}$ is given by
\[
\begin{pmatrix}
 0 & 2 & 3 & 0 \\
 2 & 6 & 9 & 1 \\
 3 & 9 & 12 & 0 \\
 0 & 1 & 0 & -2 
\end{pmatrix}. 
\]

The lattice $L\simeq U\oplus A_2$ is a primitive sublattice of the $K3$ lattice $\Lambda_{K3}=U^{\oplus 3} \oplus E_8^{\oplus 2}$. 
By definition, the lattice $\TPic{\Delta^\circ}$ is actually equal to the lattice $L$.  
Moreover, since the order of the discriminant group $A_L=L^*\slash L$ coincides with $\vert\det{L}\vert=3$, and the group $A_L$ is a finitely-generated Abelian group, we have $A_L\simeq\mathbb{Z}\slash 3\mathbb{Z}$. 

We compute $\Pic_{\Delta}$. 
The vectors $\{ v_1, \, v_5,\, v_8\}$ are linearly independent. 
Let $L'$ be the lattice that is generated by the divisors 
\[
\mathcal{B}' = \{ V_2,\, V_3,\, V_4,\, V_6,\, V_7,\, V_9,\ldots ,V_{15},\, V_{16}^{(1)},\, V_{16}^{(2)},\, V_{17}^{(1)},\, V_{17}^{(2)} \}. 
\] 
Here, $V_i = V_{i}^{(1)}+V_{i}^{(2)}$ is a restricted toric divisor for $i=17, 18$. 
By a direct computation, we have $\det L' = -3$ and $\sign L' = (1,15)$ with the aid of Mathematica. 

Note that the intersection matrix of $L'$ with respect to the basis $\mathcal{B}'$ is given by 
\[
\begin{pmatrix}
 0 & 0 & 0 & 0 & 1 & 0 & 0 & 0 & 0 & 0 & 0 & 0 & 0 & 0 & 1 & 1 \\
 0 & -2 & 0 & 1 & 0 & 0 & 0 & 0 & 0 & 0 & 0 & 0 & 0 & 0 & 0 & 0 \\
 0 & 0 & -2 & 0 & 0 & 0 & 0 & 0 & 0 & 0 & 0 & 1 & 1 & 1 & 0 & 0 \\
 0 & 1 & 0 & -2 & 1 & 0 & 0 & 0 & 0 & 0 & 0 & 0 & 0 & 0 & 0 & 0 \\
 1 & 0 & 0 & 1 & -2 & 0 & 0 & 0 & 0 & 0 & 0 & 0 & 0 & 0 & 0 & 0 \\
 0 & 0 & 0 & 0 & 0 & -2 & 1 & 0 & 0 & 0 & 0 & 0 & 0 & 0 & 0 & 0 \\
 0 & 0 & 0 & 0 & 0 & 1 & -2 & 1 & 0 & 0 & 0 & 0 & 0 & 0 & 0 & 0 \\
 0 & 0 & 0 & 0 & 0 & 0 & 1 & -2 & 1 & 0 & 0 & 0 & 0 & 0 & 0 & 0 \\
 0 & 0 & 0 & 0 & 0 & 0 & 0 & 1 & -2 & 1 & 0 & 0 & 0 & 0 & 0 & 0 \\
 0 & 0 & 0 & 0 & 0 & 0 & 0 & 0 & 1 & -2 & 1 & 0 & 0 & 0 & 0 & 0 \\
 0 & 0 & 0 & 0 & 0 & 0 & 0 & 0 & 0 & 1 & -2 & 1 & 0 & 0 & 0 & 0 \\
 0 & 0 & 1 & 0 & 0 & 0 & 0 & 0 & 0 & 0 & 1 & -2 & 0 & 0 & 0 & 0 \\
 0 & 0 & 1 & 0 & 0 & 0 & 0 & 0 & 0 & 0 & 0 & 0 & -2 & 0 & 1 & 0 \\
 0 & 0 & 1 & 0 & 0 & 0 & 0 & 0 & 0 & 0 & 0 & 0 & 0 & -2 & 0 & 1 \\
 1 & 0 & 0 & 0 & 0 & 0 & 0 & 0 & 0 & 0 & 0 & 0 & 1 & 0 & -2 & 0 \\
 1 & 0 & 0 & 0 & 0 & 0 & 0 & 0 & 0 & 0 & 0 & 0 & 0 & 1 & 0 & -2 
\end{pmatrix}. 
\]

By Lemma~\ref{PrimeDet}, the lattice $L'$ is a primitive sublattice of the $K3$ lattice $\Lambda_{K3}=U^{\oplus3}\oplus E_8^{\oplus2}$. 
By definition, the lattice $\Pic_{\Delta}$ is actually equal to the lattice $L'$. 

Since $\rk L' > 12$, there exists an even negative-definite lattice $L''$ of rank $14$ and $\det L'' = 3$ such that $L' = U\oplus L''$ holds. 
Moreover, since the order of the discriminant group $A_{L'}=L'^*\slash L'$ coincides with $\vert\det{L'}\vert=3$, and the group $A_{L'}$ is a finitely-generated Abelian group, we have $A_{L'}\simeq\mathbb{Z}\slash 3\mathbb{Z}$. 

We have seen that 
\begin{enumerate}
\item[i)] $\det L = -3 = -\det (U\oplus L')$. 
\item[ii)] $A_L\simeq\mathbb{Z}\slash 3\mathbb{Z} \simeq A_{U\oplus L'}$. 
\end{enumerate}
By Corollary \ref{orthogonal}, we conclude that the lattice $L=\TPic{\Delta^\circ}$ and the lattice $U\oplus L' = U\oplus \Pic_{\Delta}$ are orthogonal to each other in the $K3$ lattice $\Lambda_{K3}$. 

In fact, one finds $L''\simeq E_6\oplus E_8$. 

\subsection{\bf $X_{2,0}$ and $S_{17}$}\label{ML:invertibleX20}
The $S_{17}$ singularity is defined by  $f=x^6y+y^2z+z^2$, and $X_{2,0}$ by $f^T=x^6+xy^2+yz^2$. Recall that the only possible invertible deformation of $S_{17}$ is given by $F= X^6Y + Y^2Z + Z^2+W^7X$, with monomials corresponding to the lattice points $(5,5,-2)$, $(-1,0,0)$, $(-1,-2,1)$, and $(0,-1,0)$. We showed in Section~\ref{SS:invertibleX20} that the Newton polytope of $F$ is not reflexive. We take the larger polytope $\Delta$ given by 
\[
\Delta = {\rm Conv}\left\{ (-1,-2,1),\, (0,-1,0),\, (6,5,-2),\, (5,5,-2),\, (-1,2,-1)\right\}. 
\]
This choice is motivated by the consideration that one should take a reflexive polytope that contains two faces 
\[
\Phi_1 = {\rm Conv}\left\{ (-1,-2,1),\, (-1,2,-1),\, (5,5,-2)\right\},
\]
and 
\[
\Phi_2 = {\rm Conv}\left\{ (-1,-2,1),\, (-1,2,-1),\, (0,-1,0)\right\}. 
\]
By construction, the Newton polytope of $F$ is a subpolytope of $\Delta$. 
We label the lattice points on the edges of $\Delta$ as
\[
\begin{matrix}
m_1=(-1,-2,1), & m_2=(0,-1,0), & m_3=(6,5,-2), \\
m_4=(5,5,-2), & m_5=(-1,2,-1), & m_6=(3,2,-1),\\
m_7=(-1,0,0)
\end{matrix}.
\]

The polar dual polytope of $\Delta$ is given by
\[
\Delta^\circ = {\rm Conv}\left\{ (0,1,3),\, (1,-6,-12),\, (1,1,2),\, (-1,1,0),\, (0,-3,-7) \right\}. 
\]
Since the polytope $\Delta^\circ$ is an integral polytope, the polytope $\Delta$ is reflexive. 
In the \cite{sage} database of reflexive polytopes, $\Delta$ has index 760 and $\Delta^\circ$ has index 3046.
We label the lattice points on the edges of $\Delta^\circ$ as
\[
\begin{matrix}
v_1=(0,1,3), & v_2=(1,1,2), & v_3=(-1,1,0), \\
v_4=(0,-3,-7), & v_5=(1,-6,-12), & v_6=(0,1,1),\\
v_7=(0,-1,-2), & v_8=(1,0,0), & v_9=(1,-1,-2), \\
v_{10}=(1,-2,-4), & v_{11}=(1,-3,-6) , & v_{12}=(1,-4,-8), \\
v_{13}=(1,-5,-10). 
\end{matrix}
\]
The Newton polytope of the deformation $F^T=WX^6 + XY^2 + YZ^2 + W^{7}$ of $f^T=x^6 + xy^2 + yz^2$ is a subpolytope of the polar dual $\Delta^\circ$ of $\Delta$. 

Next, we compute the toric contribution $\rk L_0(\Delta)$. 

Let $\Gamma$ be the edge in $\Delta$ given by
\[
\Gamma = {\rm Conv}\left\{ (-1,-2,1),\, (-1,2,-1)\right\}. 
\]
One has its polar dual
\[
\Gamma^\circ = {\rm Conv}\left\{ (1,1,2),\, (1,-6,-12)\right\}, 
\]
and 
\[
\ell^*(\Gamma) = 1, \qquad l^*(\Gamma^\circ) = 6. 
\]
In fact, no other edge that contributes to $\rk L_0(\Delta)$. 
Thus, 
\[
\rk L_0(\Delta) = 6. 
\]

Moreover, one has
\begin{eqnarray*}
\rho(\Delta) = 5+8+6-3 = 16, & & 
\rho(\Delta^\circ) = 5+2+6-3 = 10. 
\end{eqnarray*}

We now compute $\TPic{\Delta^\circ}$. 
Let $L$ be the lattice that is generated by the divisors $D_2,\, D_5,\, D_6$ and $D_7$, where the divisor $D_i$ corresponds to the lattice point $m_i$.
By taking the new generators $\{ D_2,\, D_2+D_5,\, -D_2-D_5+D_6-D_7,\, D_7-D_2\}$, we have 
\[
L\simeq U\oplus \left(\begin{smallmatrix}-4 & 1 \\ 1 & -2\end{smallmatrix}\right).
\]
By a direct computation, we have $\det L = -7$ and $\sign L = (1,3)$. 

By Lemma~\ref{PrimeDet}, the lattice $L$ is a primitive sublattice of the $K3$ lattice $\Lambda_{K3} =U^{\oplus 3} \oplus E_8^{\oplus 2}$. 
By definition, the lattice $\TPic{\Delta^\circ}$ is actually equal to the lattice $L$.  
Moreover, since the order of the discriminant group $A_L=L^*\slash L$ coincides with $\vert\det{L}\vert=7$, and the group $A_L$ is a finitely-generated Abelian group, we have $A_L\simeq\mathbb{Z}\slash 7\mathbb{Z}$. 

We compute $\Pic_{\Delta}$. 
Let $L'$ be a lattice that is generated by divisors 
\[
\mathcal{B}' = \{ V_1,\, V_3,\, V_4,\, V_5,\, V_8^{(1)}, \ldots ,V_{13}^{(1)},\, V_8^{(2)}, \ldots , V_{13}^{(2)}\}. 
\] 
Here, $V_i = V_{i}^{(1)}+V_{i}^{(2)}$ is a restricted toric divisor for $i=8, \ldots , 13$. 
By a direct computation, we have $\det L' = -7$ and $\sign L' = (1,15)$ with the aid of Mathematica. 

By Lemma~\ref{PrimeDet}, the lattice $L'$ is a primitive sublattice of the $K3$ lattice $\Lambda_{K3}$. 
By definition, the lattice $\Pic_{\Delta}$ is actually equal to the lattice $L'$. 

Since $\rk L' > 12$, there exists an even negative-definite lattice $L''$ of rank $14$ and $\det L'' = 7$ such that $L' = U\oplus L''$ holds. 
Moreover, since the order of the discriminant group $A_{L'}=L'^*\slash L'$ coincides with $\vert\det{L'}\vert=7$, and the group $A_{L'}$ is a finitely-generated Abelian group, we have $A_{L'}\simeq\mathbb{Z}\slash 7\mathbb{Z}$. 

\bigskip
We have seen that 
\begin{enumerate}
\item[i)] $\det L = -7 = -\det (U\oplus L')$. 
\item[ii)] $A_L\simeq\mathbb{Z}\slash 7\mathbb{Z} \simeq A_{U\oplus L'}$. 
\end{enumerate}
By Corollary \ref{orthogonal}, we conclude that the lattice $L=\TPic{\Delta^\circ}$ and the lattice $U\oplus L' = U\oplus \Pic_{\Delta}$ are orthogonal to each other in the $K3$ lattice $\Lambda_{K3}$. 

In fact, using a list of \cite{Nishiyama96}, one finds 
\[
L' = U\oplus A_6\oplus E_8 . 
\]

\subsection{\bf $W_{18}$} 

The singularity $W_{18}$ is defined by $f = x^7 + y^2z + z^2$. There are two possible deformations, $F_1=X^7+Y^2Z+Z^2+W^7Y$ and $F_2=X^7+Y^2Z+Z^2+W^8X$. In Section~\ref{SS:invertibleW18}, we showed that the Newton polytope of $F_1$ is not reflexive, and although the Newton polytope of $F_2$ is reflexive, its polar dual does not contain the Newton polytope of $F_2^T$. Consider the larger lattice polytope $\Delta$ given by 
\[
\Delta = {\rm Conv}\left\{ (-3,5,-1), \, (2,-1,0),\, (-1,1,1),\, (0,-1,0),\, (0,0,1)\right\}. 
\]

The polytope $\Delta$ is reflexive and the Newton polytopes of $F_1$ and of $F_2$ are both subpolytopes of $\Delta$. 
The Newton polytope of the deformation $F_2^T=WX^7 + Y^2 + YZ^2 + W^{8}$ of $f^T=x^7 + y^2z + z^2$ is a subpolytope of the polar dual $\Delta^\circ$ of $\Delta$. 
In the \cite{sage} database of reflexive polytopes, $\Delta$ has index 760 and $\Delta^\circ$ has index 3046. These are precisely the same indices as the pair of polytopes we identified for the $X_{2,0}$ and $S_{17}$ singularity in Section~\ref{ML:invertibleX20}. Thus, the reflexive polytopes are isomorphic, and as before we have lattices given by

\[
\TPic{\Delta^\circ}\simeq U\oplus \left(\begin{smallmatrix}-4 & 1 \\ 1 & -2\end{smallmatrix}\right)
\]

\noindent and 

\[
\Pic{\Delta} = U\oplus A_6\oplus E_8 . 
\]

\subsection{\bf $W_{17}$ and $S_{1,0}$} 
The $S_{1,0}$ singularity is defined by  $f=x^5y+y^2z+z^2$, and $W_{17}$ by $f^T=x^5+xy^2+yz^2$. 
Recall that the only possible invertible deformation of the singularity $S_{1,0}$ is given by $F = X^5Y + Y^2Z + Z^2 + W^{10}$, with monomials corresponding to the lattice points $(0,1,0)$, $(-1,1,0)$, $(0,0,1)$, and $(4,-6,-3)$, respectively. 
As we have seen in the previous section, the Newton polytope $\Delta_F$ of $F$ is not reflexive. 
So, we take the larger reflexive polytopes that have $\Delta_F$ as a subpolytope. 
In fact, there are two such possibilities as we shall see below. 
In the following, for a lattice $L$, we denote by $L_{\mathcal{B}}$ the intersection matrix of $L$ with respect to the set of generators $\mathcal{B}$. 

\noindent
Case (i)\, Let the polytope $\Delta$ be given by 
\[
\Delta = {\rm Conv}\left\{ (0,1,0),\, (0,0,1),\, (2,-2,-1),\, (-2,2,-1),\, (4,-6,-3)\right\}. 
\]

By construction, the Newton polytope of $F$ is a subpolytope of $\Delta$. 
We label the lattice points on the edges of $\Delta$ as 
\[
\begin{matrix}
m_1=(0,1,0), & m_2=(0,0,1), & m_3=(2,-2,-1), \\
m_4=(-2,2,-1), & m_5=(4,-6,-3), & m_6=(3,-4,-2),\\
m_7=(2,-3,-1), & m_8=(1,-1,0), & m_9=(1,-2,-2), \\
m_{10}=(-1,1,0)
\end{matrix}.
\]

The polar dual polytope of $\Delta$ is given by
\[
\Delta^\circ = {\rm Conv}\left\{ (-1,-1,1),\, (-2,-1,-1),\, (0,-1,-1),\, (5,4,-1),\, (-1,0,-1)\right\}. 
\]
We label the lattice points on the edges of $\Delta^\circ$ as 
\[
\begin{matrix}
v_1=(-1,-1,1), & v_2=(5,4,-1), & v_3=(-1,0,-1), \\
v_4=(-2,-1,-1), & v_5=(0,-1,-1), & v_6=(2,2,-1),\\
v_7=(-1,-1,-1), & v_8=(4,3,-1), & v_9=(3,2,-1), \\
v_{10}=(2,1,-1),  & v_{11}=(1,0,-1)
\end{matrix}.
\]

Since the polytope $\Delta^\circ$ is an integral polytope, the polytope $\Delta$ is reflexive. The index of $\Delta$ in the SageMath database is 1959, while the index of $\Delta^\circ$ is 1960. 
The Newton polytope of the deformation $F^T=X^5 + XY^2 + YZ^2 + W^{10}$ of $f^T=x^5 + xy^2 + yz^2$ is a subpolytope of the polar dual $\Delta^\circ$ of $\Delta$. 

Next, we compute the toric contribution $\rk L_0(\Delta)$. 

Let $\Gamma_1$ be an edge in $\Delta$ given by
\[
\Gamma_1 = {\rm Conv}\left\{ (0,0,1),\, (-2,2,-1)\right\}. 
\]
One has its polar dual
\[
\Gamma_1^\circ = {\rm Conv}\left\{ (0,-1,-1),\, (5,4,-1)\right\},
\]
and 
\[
l^*(\Gamma_1) = 1,\qquad l^*(\Gamma_1^\circ)=4. 
\]

Let $\Gamma_2$ be an edge in $\Delta$ given by
\[
\Gamma_2 = {\rm Conv}\left\{ (0,0,1),\, (4,-6,-3)\right\}. 
\]
One has its polar dual
\[
\Gamma_2^\circ = {\rm Conv}\left\{ (-1,0,-1),\, (5,4,-1)\right\},
\]
and 
\[
l^*(\Gamma_2) = 1,\qquad l^*(\Gamma_{2}^\circ)=1. 
\]
In fact, no other edge contributes to $\rk L_0(\Delta)$. 
Thus, 
\[
\rk L_0(\Delta) = l^*(\Gamma_1) l^*(\Gamma_1^\circ) +l^*(\Gamma_2)  l^*(\Gamma_{12}^\circ) =4+1=5. 
\]

Moreover, one has
\begin{eqnarray*}
\rho(\Delta) = 11+5-3 = 13, & & 
\rho(\Delta^\circ) = 10+5-3 = 12. 
\end{eqnarray*}

We now compute $\TPic{\Delta^\circ}$. 
Let $L$ be the lattice that is generated by the divisors $D_1, D_2, D_4, D_5, D_6, D_9, D_{10}$, where the divisor $D_i$ is on a generic hypersurface corresponding to the lattice point $m_i$. 
By a direct computation, we have $\det{L}=20$ and $\sign{L}=(1,\,6)$. 
Note that the lattice $L$ contains primitively the hyperbolic lattice $U\simeq \langle D_1+D_6-D_9,\, D_2-D_9\rangle_{\mathbb{Z}}$.  
By definition, a lattice generated by the restrictions of toric divisors is primitive in the $K3$ lattice $\Lambda_{K3}$. 
We have that the discriminant group $A_L$ of $L$ is given by $A_L\simeq\mathbb{Z}\slash20\mathbb{Z}$ since the element associated to the divisor ${-}D_2 + 2 D_9 + D_{10}$ in $A_L$, is of order $20$. 
Indeed, let $\phi:\mathbb{Z}^7\to\mathbb{Q}^7$ be a map for $\mbi{x}=(x_1,\ldots,x_7)\in\mathbb{Z}^7$, 
\[
\phi(\mbi{x}):=\mbi{x}.L_{\mathcal{B}}^{-1}. 
\]
We can show directly that 
\begin{eqnarray*}
20\cdot\phi(0,\,0,\,1,\,0,\,0,\,0,\,1) & \in & \mathbb{Z}^7, \\
k\cdot\phi(0,\,0,\,1,\,0,\,0,\,0,\,1) & \in & \mathbb{Q}^7\backslash\mathbb{Z}^7, \, k=2,4,5,10. \\
\end{eqnarray*}


We calculate $\Pic_{\Delta}$. 
Let $L'$ be a lattice that is generated by divisors on a generic hypersurface
\[
\mathcal{B}'=\{ V_1, \, V_2, \, V_4, \, V_6^{(1)},\, V_6^{(2)}, \, V_8^{(1)},\, V_9^{(1)},\, V_{10}^{(1)},\, V_{11}^{(1)}, \, V_8^{(2)},\, V_9^{(2)},\, V_{10}^{(2)},\, V_{11}^{(2)} \}. 
\]
Here, $V_i=V_i^{(1)}+V_i^{(2)}$ is a restricted toric divisor for $i=6,8,9,10,11$. 
By a direct computation, we have $\det{L'}=20$ and $\sign{L'}=(1,\,12)$. 
We have  that the discriminant group $A_{L'}$ of $L'$ is given by $A_{L'}\simeq\mathbb{Z}\slash20\mathbb{Z}$ since the element associated to the divisor $V_6^{(1)} + V_8^{(1)}$ in $A_{L'}$, is of order $20$. 
Indeed, let $\psi:\mathbb{Z}^{13}\to\mathbb{Q}^{13}$ be a map for $\mbi{x}=(x_1,\ldots,x_{13})\in\mathbb{Z}^{13}$, 
\[
\psi(\mbi{x}):=\mbi{x}.{L'_{\mathcal{B}'}}^{-1}. 
\]
We can show directly that 
\begin{eqnarray*}
20\cdot\phi(0,\,0,\,0,\,1,\,0,\,1,\,0,\ldots,\,0) & \in & \mathbb{Z}^{13}, \\
k\cdot\phi(0,\,0,\,0,\,1,\,0,\,1,\,0,\ldots,\,0) & \in & \mathbb{Q}^{13}\backslash\mathbb{Z}^{13}, \, k=2,4,5,10. 
\end{eqnarray*}


We claim that the lattice $L'$ is a primitive lattice of the $K3$ lattice $\Lambda_{K3}=U^{\oplus3}\oplus E_8^{\oplus2}$,\, which is of signature $(l_+,\,l_-)=(3,19)$. 
We have $l_+-t_+=3-1=2\geq0$, $l_--t_-=19-12=7\geq0$, and $\rk{\Lambda_{K3}}-\rk{L'}=22-13=9>1=l(A_{L'})$. 
Thus by Corollary \ref{primitive}, the lattice $L'$ is indeed a primitive sublattice of the $K3$ lattice. 

We have seen that 
\begin{enumerate}
\item[i)] $\det{L}=20=-\det{(U\oplus L')}$. 
\item[ii)] $A_{L}\simeq\mathbb{Z}\slash20\mathbb{Z}\simeq A_{L'}$. 
\end{enumerate}

By Corollary \ref{orthogonal}, we conclude that the lattices $L=\TPic{\Delta^*}$ and $U\oplus L' = U\oplus \Pic_{\Delta}$ are orthogonal to each other in the $K3$ lattice $\Lambda_{K3}$. \\

\noindent
Case (ii)\, Let the polytope $\Delta$ be given by

 $\Delta=\Delta_{(2,3,5,10)}={\rm Conv}\left\{ (1,0,0),\, (0,1,0),\, (0,0,1),\, (-2,2,-1),\, (4,-6,-3)\right\}$. \\

By construction, the Newton polytope of $F$ is a subpolytope of $\Delta$. 
We label the lattice points on the edges of $\Delta$ as 
\[
\begin{matrix}
m_1=(0,1,0), & m_2=(0,0,1), & m_3=(1,0,0), \\
m_4=(-2,2,-1), & m_5=(4,-6,-3), & m_6=(3,-4,-2),\\
m_7=(2,-2,-1), & m_8=(2,-3,-1), & m_9=(1,-2,-2), \\
m_{10}=(-1,1,0)
\end{matrix}.
\]

The polar dual polytope of $\Delta$ is given by
\[
\Delta^\circ = {\rm Conv}\left\{ (-1,-1,1),\, (-1,-1,-1),\, (0,-1,-1),\, (5,4,-1),\, (-1,0,-1)\right\}. 
\]
We label the lattice points on the edges of $\Delta^\circ$ as 
\[
\begin{matrix}
v_1=(-1,-1,1), & v_2=(5,4,-1), & v_3=(-1,0,-1), \\
v_4=(-1,-1,-1), & v_5=(0,-1,-1), & v_6=(2,2,-1),\\
v_7=(-1,-1,0), & v_8=(4,3,-1), & v_9=(3,2,-1), \\
v_{10}=(2,1,-1),  & v_{11}=(1,0,-1)
\end{matrix}.
\]

Since the polytope $\Delta^\circ$ is an integral polytope, the polytope $\Delta$ is reflexive. 
The Newton polytope of the deformation $F^T=X^5 + XY^2 + YZ^2 + W^{10}$ of $f^T=x^5 + xy^2 + yz^2$ is a subpolytope of the polar dual $\Delta^\circ$ of $\Delta$.

The edges $\Gamma_1$ and $\Gamma_2$ defined in Case (i) are also edges of $\Delta$ and there is no other contribution to $\rk L_0(\Delta)$. 
Therefore, we have $\rk L_0(\Delta)=5$. 

Moreover, one has
\begin{eqnarray*}
\rho(\Delta) = 11+5-3 = 13, & & 
\rho(\Delta^\circ) = 10+5-3 = 12. 
\end{eqnarray*}

We now compute $\TPic{\Delta^\circ}$. 
Let $L$ be the lattice that is generated by the divisors $D_1, D_2, D_3, D_7, D_6, D_5, D_9$, where the divisor $D_i$ is on a generic hypersurface corresponding to the lattice point $m_i$. 
By a direct computation, we have $\det{L}=20$ and $\sign{L}=(1,\,6)$. 
Note that the lattice $L$ contains primitively the hyperbolic lattice $U\simeq\langle D_1+D_6-D_7,\, D_1+D_2-2D_3\rangle_{\mathbb{Z}}$. 

By definition, a lattice generated by the restrictions of toric divisors is primitive in the $K3$ lattice $\Lambda_{K3}$. 
We have that the discriminant group $A_L$ of $L$ is given by $A_L\simeq\mathbb{Z}\slash20\mathbb{Z}$ since the element associated to the divisor $D_6 + D_7$ in $A_L$ is of order $20$. 
Indeed, let $\phi:\mathbb{Z}^7\to\mathbb{Q}^7$ be a map for $\mbi{x}=(x_1,\ldots,x_7)\in\mathbb{Z}^7$, 
\[
\phi(\mbi{x}):=\mbi{x}.L_{\mathcal{B}}^{-1}. 
\]
We can show directly that 
\begin{eqnarray*}
20\cdot\phi(0,\,0,\,0,\,0,\,1,\,0,\,0) & \in & \mathbb{Z}^7, \\
k\cdot\phi(0,\,0,\,0,\,0,\,1,\,0,\,0) & \in & \mathbb{Q}^7\backslash\mathbb{Z}^7, \, k=2,4,5,10. \\
\end{eqnarray*}

Note that the lattices $L$ in Cases (i) and (ii) are isometric. 

We calculate $\Pic_{\Delta}$. 
Let $L'$ be a lattice that is generated by divisors on a generic hypersurface
\[
\mathcal{B}'=\{ V_1, \, V_2, \, V_4, \, V_6^{(1)},\, V_6^{(2)}, \, V_8^{(1)},\, V_9^{(1)},\, V_{10}^{(1)},\, V_{11}^{(1)}, \, V_8^{(2)},\, V_9^{(2)},\, V_{10}^{(2)},\, V_{11}^{(2)} \}. 
\]

Here, $V_i=V_i^{(1)}+V_i^{(2)}$ is a restricted toric divisor for $i=6,8,9,10,11$. 

By a direct computation, we have $\det{L'}=20$ and $\sign{L'}=(1,\,12)$. 
We have  that the discriminant group $A_{L'}$ of $L'$ is given by $A_{L'}\simeq\mathbb{Z}\slash20\mathbb{Z}$ since the element associated to the divisor $V_6^{(1)} + V_8^{(1)}$ in $A_{L'}$ is of order $20$. 
Indeed, let $\psi:\mathbb{Z}^{13}\to\mathbb{Q}^{13}$ be a map for $\mbi{x}=(x_1,\ldots,x_{13})\in\mathbb{Z}^{13}$, 
\[
\psi(\mbi{x}):=\mbi{x}.{L'_{\mathcal{B}'}}^{-1}. 
\]
We can show directly that 
\begin{eqnarray*}
20\cdot\phi(0,\,0,\,0,\,1,\,0,\,1,\,0,\ldots,\,0) & \in & \mathbb{Z}^{13}, \\
k\cdot\phi(0,\,0,\,0,\,1,\,0,\,1,\,0,\ldots,\,0) & \in & \mathbb{Q}^{13}\backslash\mathbb{Z}^{13}, \, k=2,4,5,10. 
\end{eqnarray*}

We claim that the lattice $L'$ is a primitive lattice of the $K3$ lattice $\Lambda_{K3}=U^{\oplus3}\oplus E_8^{\oplus2}$, which is of signature $(l_+,\,l_-)=(3,19)$. 
We have $l_+-t_+=3-1=2\geq0$, $l_--t_-=19-12=7\geq0$ and $\rk{\Lambda_{K3}}-\rk{L}=22-13=9>1=l(A_{L})$. 
By Corollary \ref{primitive}, the lattice $L'$ is indeed a primitive sublattice of the $K3$ lattice.

We have seen that 
\begin{enumerate}
\item[i)] $\det{L}=20=-\det{(U\oplus L')}$. 
\item[ii)] $A_{L}\simeq\mathbb{Z}\slash20\mathbb{Z}\simeq A_{L'}$. 
\end{enumerate}

By Corollary \ref{orthogonal}, we conclude that the lattice $L=\TPic{\Delta^\circ}$ and the lattice $U\oplus L' = U\oplus \Pic_{\Delta}$ are orthogonal to each other in the $K3$ lattice $\Lambda_{K3}$. 

We can easily observe that in cases (i) and (ii), although the polytopes are not isomorphic, but the lattices $\TPic{\Delta^\circ}$ and $\Pic{\Delta}$ are respectively isometric. 

%
%
%
%
%

\subsection{\bf $U_{16}$}

The singularity $U_{16}$ is defined by $f = x^5 + y^2z + yz^2$, and consider a deformation $F=X^5 + Y^2Z + YZ^2 + W^6X$, and let the polytope $\Delta$ be given by 
\[
\Delta:=\Delta_F = {\rm Conv}\left\{  (1,0,-1), \, (0,0,1),\, (-2,-2,3),\,  (0,1,-1) \right\}. 
\]
By construction, the Newton polytope of $F$ is a subpolytope of $\Delta$. 
We label the lattice points on the edges of $\Delta$ as 
\[
\begin{matrix}
m_1 = (1,0,-1), & m_2=(0,0,1), & m_3=(-2,-2,3), \\
m_4=(0,1,-1), & m_5=(-1,-1,2). 
\end{matrix}
\]

The polar dual poltope of $\Delta$ is given by 
\[
\Delta^\circ = {\rm Conv}\left\{(-2,-2,-1),\, (-2,1,-1),\, (4,4,5),\, (1,-2,-1) \right\}. 
\]
We label the lattice points on the edges of $\Delta^\circ$ as 
\[
\begin{matrix}
v_1 = (-2,-2,-1), & v_2 = (-2,1,-1), & v_3 = (4,4,5), & v_4=(1,-2,-1), \\
v_5 = (-2,-1,-1), & v_6 = (-2,0,-1), & v_7 = (0,2,1), & v_8=(2,3,3), \\
v_9 = (-1,-2,-1), & v_{10} = (0,-2,-1), & v_{11} = (3,2,3), & v_{12}=(2,0,1), \\
v_{13} = (-1,-1,0), & v_{14} = (0,0,1), & v_{15} = (1,1,2), & v_{16}=(2,2,3), \\
v_{17} = (3,3,4), & v_{18} = (-1,0,-1), & v_{19} = (0,-1,-1)
\end{matrix}. 
\]
Since the polytope $\Delta^\circ$ is an integral polytope, the polytope $\Delta$ is reflexive. 
The Newton polytope of the deformation $F^T=WX^5 + Y^2Z + YZ^2 + W^6$ of $f^T=x^5 + y^2z + yz^2$ is a subpolytope of the polar dual $\Delta^\circ$ of $\Delta$. 
As we have seen in the previous section, the toric contribution is $\rk L_0(\Delta) = 2$. 


Moreover, one has
\begin{eqnarray*}
\rho(\Delta)  = 19+2-3 = 18, & & 
\rho(\Delta^\circ)  = 5+2-3 = 4. 
\end{eqnarray*}

We compute $\TPic{\Delta^\circ}$. 
A collection $\{ m_1, \, m_2,\, m_5\}$ is a linearly-independent vectors. 
Let $L$ be a lattice that is generated by divisors $D_3$ and $D_4$.  
By taking the new generators $\{ D_3,\, -D_3+D_4\}$, we have $L\simeq U(3)$. 
It is known that the discriminant group of $U(3)$ is isomorphic to $\mathbb{Z}\slash 9\mathbb{Z}$. 
By definition, the lattice $\TPic{\Delta^\circ}$ is actually equal to the lattice $L$.  

We compute $\Pic_{\Delta}$. 
A collection $\{ v_1, \, v_5,\, v_9\}$ is a linearly-independent vectors. 
Let $L'$ be a lattice that is generated by divisors 
\[
\mathcal{B}' = \{ V_2,\, V_3,\, V_4,\, V_6,\, V_7,\, V_8,\, V_{10},\ldots ,V_{17},\, V_{18}^{(1)},\, V_{18}^{(2)},\, V_{19}^{(1)},\, V_{19}^{(2)} \}. 
\] 
Here, $V_i = V_{i}^{(1)}+V_{i}^{(2)}$ is a restricted toric divisor for $i=18, 19$. 

By a direct computation, we have $\det L' = -9$ and $\sign L' = (1,17)$ with the aid of Mathematica. 

Since $\rk L' > 12$, there exists an even negative-definite lattice $L''$ of rank $16$ and $\det L'' = 9$ such that $L' = U\oplus L''$ holds. 
Moreover, set 
${L'}_{\mathbb{B'}}$ to be the intersection matrix of $L'$ with respect to the basis $\mathcal{B}'$, and by defining a map $\phi : \mathbb{Z}^{18}\to\mathbb{Q}^{18}$ 
\[
\phi(x_1, \ldots, x_{18} ; s) := s \cdot( x_1 \, \cdots \, x_{18} ) .{L'}_{\mathbb{B'}}^{-1}, 
\]
and one can see that 
\[
\phi(x_1, \ldots, x_{18} ; 9)\in\mathbb{Z}. 
\]

This means that there exists an element of order $9$ in the discriminant group of $L'$. 
Therefore, $L'^* \slash L' \simeq \mathbb{Z}\slash 9\mathbb{Z}$, so $L'^* \slash L' \simeq (L\oplus U)^* \slash (L\oplus U)$. 
Thus, the lattice $L'$ is also a primitive sublattice of $\Lambda_{K3}$, and $L'\simeq (L\oplus U)^\perp$ holds. 

\bigskip
We have seen that 
\begin{enumerate}
\item[i)] $\det L = -9 = -\det (U\oplus L')$. 
\item[ii)] $A_L\simeq\mathbb{Z}\slash 9\mathbb{Z} \simeq A_{U\oplus L'}$. 
\end{enumerate}
By Corollary \ref{orthogonal}, we conclude that the lattices $L=\TPic{\Delta^\circ}$ and $U\oplus L' = U\oplus \Pic_{\Delta}$ are orthogonal to each other in the $K3$ lattice $\Lambda_{K3}$. 
\QED

\hfill Makiko Mase, \\

\hfill University of Mannheim, \\

\hfill mtmase@arion.ocn.ne.jp \\

\hfill Ursula Whitcher, \\

\hfill Mathematical Reviews (American Mathematical Society)\\

\hfill uaw@umich.edu\\

\end{document}